\documentclass[amsthm]{elsart}

\usepackage{yjsco}
\usepackage{amsfonts}
\usepackage{amssymb}
\usepackage{amsmath}
\usepackage{graphicx}

\usepackage[numbers,square,sort,comma]{natbib}

\newtheorem{example}{Example}
\newtheorem{defi}{Definition}

\begin{document}
\begin{frontmatter}

\title{On the Complexity of Hilbert Refutations for Partition}

\vspace{-25pt}

\author{S. Margulies}
\address{Department of Mathematics,
Pennsylvania State University,
State College, PA
}
\ead{margulies@math.psu.edu}

\author{S. Onn}
\address{Industrial Engineering \& Management,
Technion - Israel Institute of Technology,
Haifa, Israel
}
\ead{onn@ie.technion.ac.il}

\author{D.V. Pasechnik}
\address{School of Physical and Mathematical Sciences,
Nanyang Technological University,
Singapore
}
\ead{dima@ntu.edu.sg}


\begin{abstract}
Given a set of integers $W$, the \textsc{Partition} problem determines whether $W$ can be divided into two disjoint subsets with equal sums. We model the \textsc{Partition} problem as a system of polynomial equations, and then investigate the complexity of a Hilbert's Nullstellensatz refutation, or certificate, that a given set of integers is not partitionable. We provide an explicit construction of a minimum-degree certificate, and then demonstrate that the \textsc{Partition} problem is equivalent to the determinant of a carefully constructed matrix called the partition matrix. In particular, we show that the determinant of the partition matrix is a polynomial that factors into an iteration over all possible partitions of $W$. 
\end{abstract}


\end{frontmatter}

\section{Introduction} 

The NP-complete problem \textsc{Partition} \cite{garey_and_johnson} is the question of deciding whether or not a given set of integers $W = \{w_1,\ldots,w_n\}$ can be broken into two sets, $I$ and $W \setminus I$, such that the sums of the two sets are equal, or that $\sum_{w \in I}w = \sum_{w \in W \setminus I} w$. Since it is widely believed that $\text{NP} \neq \text{coNP}$, it is interesting to study various types of \emph{refutations}, or certificates for the \emph{non}-existence of a partition in a given set $W$. 

In this paper, we study the certificates provided by Hilbert's Nullstellensatz (see \cite{Alo,AT,DLMO,Lov,Onn} and references therein). Given an algebraically-closed field $\mathbb{K}$ and a set of polynomials $f_1,\ldots,f_s \in \mathbb{K}[x_1,\ldots,x_n]$, Hilbert's Nullstellensatz states that the system of polynomial equations $f_1 = f_2 = \cdots = f_s = 0$ has \emph{no} solution if and only if there exist polynomials $\beta_1,\ldots,\beta_s \in \mathbb{K}[x_1,\ldots,x_n]$ such that
$1 = \sum_{i=1}^{s}\beta_if_i~.$
We measure the complexity of a given certificate in terms of the size of the $\beta$ coefficients, since these are the unknowns we must discover in order to demonstrate the \emph{non}-existence of a solution to $f_1 = f_2 = \cdots = f_s = 0$. Thus, we measure the degree of a Nullstellensatz certificate as $d = \max\{\deg(\beta_1),\ldots, \deg(\beta_s)\}$.

There is a well-known connection between Hilbert's Nullstellensatz and a particular sequence of linear algebra computations.
These sequences have been studied from both a theoretical perspective \cite{buss, DLMO}, and a computational perspective \cite{susan_issac, susan_jsc}. When the polynomial ideal contains $x_i^2 - x_i$ for each variable (thus forcing the variety to contain only 0/1 points), these sequences have also been explored as algebraic proof systems \cite{beame_null, clegg, impag, raz}. Additionally, D. Grigoriev demonstrates a linear lower bound for the knapsack problem in \cite{grigoriev} (see also \cite{GHP}), and Buss and Pitassi \cite{buss} show that a polynomial system loosely based upon the ``pigeon-hole principle" requires a $\lfloor \log n\rfloor - 1$ Nullstellensatz degree certificate. However, when the system of polynomial equations $f_1,\ldots,f_s$ models an NP-complete problem, the degree $d$ is likely to grow at least linearly with the size of the underlying NP-complete instance \cite{susan_thesis}. In other words, as long as $\text{P} \neq \text{NP}$, the certificates should be hard to find (i.e., the size of the linear systems involved should be exponential in the size of the underlying instance), and as long as $\text{NP} \neq \text{coNP}$, the certificates should be hard to verify (i.e., the certificates should contain an exponential number of monomials). 

For example, consider the NP-complete problem of finding an independent set of size $k$ in a graph $G$. Recall that an independent set is a set of pairwise non-adjacent vertices. This problem was modeled by Lov\'asz \cite{Lov} as a system of polynomial equations as follows:

\begin{minipage}{0.6\linewidth}
\begin{align*}
x_i^2-x_i&=0~, \mbox{ for every vertex } i \in V(G)~, \\
x_ix_j&=0~, \hbox{ for every edge } (i,j) \in E(G)~, 
\end{align*}
\end{minipage}
\begin{minipage}{0.2\linewidth}
\begin{align*}
\text{and} \quad -k + \sum_{i=1}^n x_i &=0~.
\end{align*}
\end{minipage}

\vspace{5pt}
Clearly, this system of polynomial equations has a solution if and only the underlying graph $G$ has an independent of size $k$. For example, consider the Tur\'an graph $T(5,3)$.
By inspection, we see that size of the largest independent set in $T(5,3)$ is two. Therefore, there is \emph{no} independent set of size three, and using the connection between Hilbert's Nullstellensatz and linear algebra (described more thoroughly in Sec \ref{sec_pm}), the authors of \cite{DLMO} produce the following certificate:

\begin{minipage}{0.15\linewidth}
\begin{center}
\includegraphics[scale=0.20, trim=100 580 0 90]{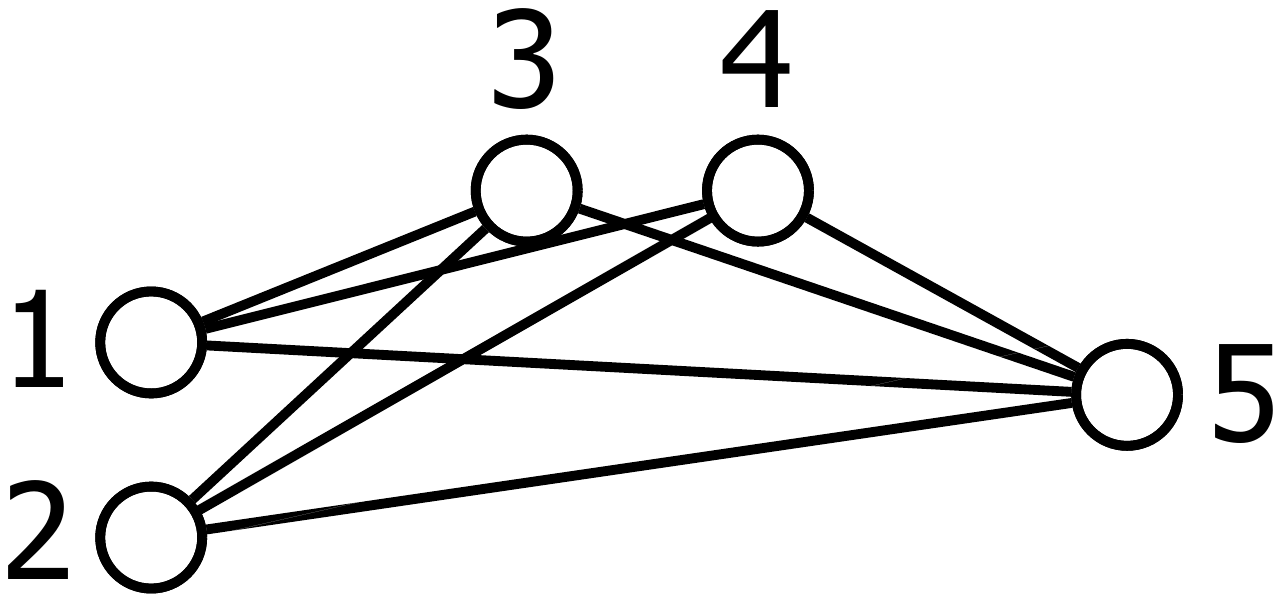}
Tur\'an graph $T(5,3)$
\end{center}
\end{minipage}
\begin{minipage}{0.85\linewidth}
\footnotesize{
\begin{align*}
& \hspace{14pt} \bigg(\frac{1}{3}x_4 + \frac{1}{3}x_2 + \frac{1}{3}\bigg)x_1x_3 + \bigg(\frac{1}{3}x_2 + \frac{1}{3}\bigg)x_1x_4 + \bigg(\frac{1}{3}x_2 + \frac{1}{3}\bigg)x_1x_5 +  \bigg(\frac{1}{3}x_4 + \frac{1}{3}\bigg)x_2x_3  +\\
& \hspace{14pt}\bigg(\frac{1}{3}\bigg)x_2x_4 + \bigg(\frac{1}{3}\bigg)x_2x_5 + \bigg(\frac{1}{3}x_4 + \frac{1}{3}\bigg)x_3x_5 + \bigg(\frac{1}{3}\bigg)x_4x_5 + \bigg(\frac{1}{3}x_2 + \frac{1}{6}\bigg)(x_1^2 - x_1) + \\
&\hspace{14pt} \bigg(\frac{1}{3}x_1 + \frac{1}{6}\bigg)(x_2^2 - x_2) + \bigg(\frac{1}{3}x_4 + \frac{1}{6}\bigg)(x_3^2 - x_3) + \bigg(\frac{1}{3}x_3 + \frac{1}{6}\bigg)(x_4^2 - x_4) + \bigg(\frac{1}{6}\bigg)(x_5^2 - x_5) +
\end{align*}}\normalsize
\end{minipage}
\begin{align*}
 \underbrace{\bigg(-\frac{1}{3}\big(x_1x_2 + x_3x_4\big) - \frac{1}{6}\big(x_1 + x_2 + x_3 + x_4 + x_5\big) - \frac{1}{3} \bigg)}_{\beta_1}(x_1 +  x_2 + x_3 + x_4 + x_5 - 3) &= 1~.
\end{align*}
The combinatorial interpretation of this algebraic identity is unexpectedly clear: the size of the largest independent set is the degree of the Nullstellensatz certificate (i.e., the largest monomial $x_1x_2$ corresponds to the maximum independent set formed by vertices $\{1,2\}$), and the coefficient $\beta_1$ contains one monomial for each independent set in $G$. The combinatorial interpretation of these certificates is proven  in \cite{DLMO} by De Loera et al. only in terms of monomials: the relationship between the numbers  such as $1/3$ and $1/6$ and the independent sets of the underlying graph is not clear. 

In this paper, we model the \textsc{Partition} problem as a system of polynomial equations, and then present a combinatorial interpretation of an associated minimum-degree Nullstellensatz certificate. However, the focus of our combinatorial interpretation is not only on the relationship between partitions and monomials, but also on the relationship between partitions and numeric coefficients (i.e., the numbers $1/3$ and $1/6$). In Section \ref{sec_part_model}, we present an algebraic model of the partition problem and describe a minimum-degree Nullstellensatz certificate. In Section \ref{sec_pm}, we describe the connection between Hilbert's Nullstellensatz and linear algebra, leading to the construction of a \emph{square} system of linear equations, forming what we  call the \emph{partition matrix}. In Section \ref{sec_det}, we prove our main result: the determinant of the partition matrix represents a brute-force iteration over \emph{all} the possible partitions of the set $W$, a polynomial we refer to as the \emph{partition polynomial}.

We conclude our introduction with an example. Let $W = \{w_1,w_2,w_3, w_4\}$, and we see that the determinant of the associated \emph{partition matrix} is as follows:

\begin{minipage}{0.5\linewidth}
\[
\det\left(\left[ \footnotesize{ \begin{array}{cccccccc}
\vspace{-5pt}
w_4 & w_3 & w_2 & w_1 & 0 & 0 & 0 & 0\\ 
\vspace{-5pt}
w_3 & w_4 & 0 & 0 & w_2 & w_1 & 0 & 0\\
\vspace{-5pt}
w_2 & 0 & w_4 & 0 & w_3 & 0 & w_1 & 0\\
\vspace{-5pt}
w_1 & 0 & 0 & w_4 & 0 & w_3 & w_2 & 0\\
\vspace{-5pt}
0 & w_2 & w_3 & 0 & w_4 & 0 & 0 & w_1\\
\vspace{-5pt}
0 & w_1 & 0 & w_3 & 0 & w_4 & 0 & w_2\\
\vspace{-5pt}
0 & 0 & w_1 & w_2 & 0 & 0 & w_4 & w_3\\
\vspace{-5pt}
0 & 0 & 0 & 0 & w_1 & w_2 & w_3 & w_4\\
\vspace{-10pt}
\phantom{=}
\end{array} \normalsize} \right]\right)=~
\]
\end{minipage}
\hspace{-17pt}\begin{minipage}{0.5\linewidth}
\footnotesize{
\begin{align*}
&\phantom{=}(w_1+w_2+w_3+w_4)(-w_1+w_2+w_3+w_4)\\
&\phantom{=}(w_1-w_2+w_3+w_4)(w_1+w_2-w_3+w_4)\\
&\phantom{=}(-w_1+w_2-w_3+w_4)(-w_1-w_2+w_3+w_4)\\
&\phantom{=}(w_1-w_2-w_3+w_4)(-w_1-w_2-w_3+w_4)~.
\end{align*}}\normalsize
\end{minipage}

\vspace{5pt}
Thus, the determinant of the \emph{partition matrix} is indeed a brute-force iteration over every possible partition of $W$: the \emph{partition polynomial}.

\section{\large{Partitions and a System of Polynomial Equations}} \label{sec_part_model} The \textsc{Partition} problem determines if a given set of integers $W = \{w_1,\ldots, w_{n}\}$ can be divided into two sets, $I$ and $W \setminus I$ such that $\sum_{w \in I}w = \sum_{w \in W \setminus I} w$. In this section, we describe a system of polynomial equations that models this question, and discuss the degree and monomials in an associated minimum-degree Nullstellensatz certificate.
\begin{prop} \label{prop_enc_part} Given a set of integers $W = \{w_1,\ldots, w_{n}\}$, the following system of polynomial equations
\begin{align*}
x_{i}^2 - 1 = 0~, & \quad \text{for $1 \leq i \leq n$}~,  \quad \quad \text{and} \quad \sum_{i = 1}^{n}w_ix_i =0~.
\end{align*}
has a solution if and only if there exists a partition of $W$ into two sets, $I \subseteq W$ and $W \setminus I$, such that $\sum_{w \in I}w = \sum_{w \in W \setminus I} w$~.
\end{prop}
\noindent \textit{Proof:} The variables $x_i$ can take on the values of $\pm 1$. Thus, we relate partitions to solutions by placing integers $w_i$ with $+1$ $x_i$ values on one side of the partition and integers $w_i$ with $-1$ $x_i$ values on the other. \hfill $\Box$\\

Let $[n]$ denote the set of integers $\{1,\ldots, n\}$ and let $S^n_k$ denote the set of $k$-subsets of $[n]$. For $S \in S^n_k$, let $x^{S}$ denote the corresponding square-free monomial of degree $k$ in $n$ variables. For example, given $S = \{1,3,4\} \subseteq [5]$, the corresponding monomial $x^S = x_1x_3x_4$. Additionally, let $S^{n \setminus i}_k$ denote the $k$-subsets of  $[n] \setminus i$~.
\begin{thm} \label{thm_min_cert} Given a set of non-partitionable integers $W=\{w_1,\ldots,w_n\}$ encoded as a system of polynomial equations according to Prop. \ref{prop_enc_part}, there exists a minimum-degree Nullstellensatz certificate for the \emph{non}-existence of a partition of $W$ as follows:
\begin{align*}
1 &= \sum_{i=1}^n\Big(\sum_{ \stackrel{k \text{~even}}{k \leq n - 1}}\sum_{S \in S^{n \setminus i}_k}c_{i,S}x^S\Big)(x_i^2 - 1) + \Big(\sum_{ \stackrel{k \text{~odd}}{k \leq n}}\sum_{S \in S^n_k}b_Sx^S\Big)\Big(\sum_{i=1}^nw_ix_i\Big)~.
\end{align*}
Moreover, every Nullstellensatz certificate for the system of equations defined by Prop. \ref{prop_enc_part} contains one monomial for each of the odd parity subsets of each $S^n_k$, and one monomial for each of the even parity subsets of each $S^{n \setminus i}_k$.
\end{thm}

Via Thm. \ref{thm_min_cert}, we see that the degree of the certificate is $n$ for $n$ odd, and $n-1$ for $n$ even. Furthermore, by considering the monomials present in the certificate as identifying the integers present on \emph{one} side of a partition, we see that the monomials represent a brute-force iteration over every possible partition of $W$. We note that we identify the constant terms $c_{i,\emptyset}$ with the case of placing every integer on one side of the partition and the empty set on the other. Thus, this result is similar to the independent set result (De Loera et al., \cite{DLMO}) reviewed in the introduction. However, in this paper, we are interested not only in a combinatorial interpretation of the monomials, but also in a combinatorial interpretation of the unknowns $c_{i,S}, b_S$.

The proof of Thm. \ref{thm_min_cert} is virtually identical to the proof of the independent set result described in \cite{DLMO}, with no new techniques or insights. The essential strategy of the proof is to consider an arbitrary minimum-degree Nullstellensatz certificate, and then ``reduce" the certificate to a version containing only square-free monomials by adding and subtracting different polynomials from the certificate (or taking the monomials modulo the ideal). It is then possible to demonstrate that every square-free monomial representing a partition must be present in the certificate; otherwise, the certificate would require an infinite chain of higher and higher degree square-free monomials in order to simplify to 1. Since the technical details of this strategy were carefully presented in \cite{DLMO}, we omit the formal proof here and simply state Thm. \ref{thm_min_cert} as a result.
\begin{example}\label{ex_cert}
\emph{
The set of integers $W = \{1, 3, 5, 2\}$ is \emph{not} partitionable. We encode this problem as a system of polynomial equations as follows:
\begin{align*}
x_1^2 - 1 = 0~, \quad x_2^2 - 1=0~, \quad x_3^3 - 1=0~, \quad x_4^2 - 1=0~,\quad x_1 + 3x_2 + 5x_3 + 2x_4 = 0~.
\end{align*}
Since $W$ is \emph{not} partitionable, this system of equations has \emph{no} solution, and a Nullstellensatz certificate exists. Here is the minimum-degree certificate described by Thm. \ref{thm_min_cert}:
\footnotesize{
\begin{align*}
1 &= \bigg(-\frac{155}{693}+\frac{842}{3465}x_2x_3-\frac{188}{693}x_2x_4+\frac{908}{3465}x_3x_4\bigg)(x_1^2-1)+ \bigg(-\frac{1}{231}+\frac{842}{1155}x_1x_3-\frac{188}{231}x_1x_4\\
&\phantom{=} +\frac{292}{1155}x_3x_4\bigg)(x_2^2-1)
+ \bigg(-\frac{467}{693}+\frac{842}{693}x_1x_2+\frac{908}{693}x_1x_4+\frac{292}{693}x_2x_4 \bigg)(x_3^2-1)+ \bigg(-\frac{68}{693}-\frac{376}{693}x_1x_2\\ 
&\phantom{=} +\frac{1816}{3465}x_1x_3+\frac{584}{3465}x_2x_3 \bigg)(x_4^2-1)+ \bigg(\frac{155}{693}x_1+\frac{1}{693}x_2+\frac{467}{3465}x_3+\frac{\mathbf{34}}{\mathbf{693}}x_4-\frac{842}{3465}x_1x_2x_3\\
&\phantom{=} +\frac{188}{693}x_1x_2x_4-\frac{908}{3465}x_1x_3x_4-\frac{292}{3465}x_2x_3x_4\bigg)(x_1+3x_2+5x_3+2x_4)~.
\end{align*}}\normalsize
Note that the coefficient for $(x_1^2 - 1)$ contains only even degree monomials that do \emph{not} contain $x_1$ (similarly for $(x_2^2-1)$, etc.) and that the coefficient for $(x_1+3x_2+5x_3+2x_4)$ contains every possible odd degree monomial in four variables. The combinatorial interpretation of a number such as $34/693$ is explicitly demonstrated in Ex. \ref{ex_b4}.\hfill $\Box$
}
\end{example}

\section{\large{The Partition Matrix: Definition and Properties}} \label{sec_pm} In this section, we explore the well-known connection between Hilbert's Nullstellensatz and linear algebra, in terms of the minimum-degree certificate defined in Thm. \ref{thm_min_cert}:
\begin{align*}
1 &= \sum_{i=1}^n\Big(\sum_{ \stackrel{k \text{~even}}{k \leq n - 1}}\sum_{S \in S^{n \setminus i}_k}c_{i,S}x^S\Big)(x_i^2 - 1) + \Big(\sum_{ \stackrel{k \text{~odd}}{k \leq n}}\sum_{S \in S^n_k}b_Sx^S\Big)\Big(\sum_{i=1}^nw_ix_i\Big)~.
\end{align*}
We begin by defining \emph{graded reverse lexicographic order}. We then construct a $2^{n-1} \times 2^{n-1}$ square system of linear equations containing only the unknowns $b$. When ordered according to \emph{graded reverse lexicographic order}, this square matrix is known as the \emph{partition matrix}. We next prove a series of properties of the partition matrix (including symmetry), and conclude by expressing the partition matrix as the sum of a very specific set of permutation matrices. The properties of these permutation matrices allow for a simple and elegant proof of the main result in Section \ref{sec_det}.

\subsection{Graded Reverse Lexicographic Order as a Tree}  \label{ssec_ord} Since we are dealing only with square-free monomials, we define \emph{graded reverse lexicographic order} (denoted $\succeq_D$) as follows. Given $S \in S^n_k$, we represent $S$ as a vector in $\{0,1\}^n$ (denoted $v_S$) by setting $v_S[i] = 1$ if $i \in S$ and $v_S[i] = 0$ otherwise. For example, let $S = \{2,3,7\} \in S^7_3$. Then $v_S = \{0,1,1,0,0,0,1\}$. Given distinct $S \in S^n_k$ and $S' \in S^n_{k'}$, then $S \succeq_D S'$ in two cases: 1) if $k > k'$, or 2) if $k = k'$ and the \emph{right-most nonzero entry} of $v_S - v_{S'}$ is negative. For example, $\{2,3,4,5\} \succeq_D \{1,2,5\}$, and $\{2,3\} \succeq_D \{1,4\}$. 

In order to prove specific properties of the partition matrix, we use a slightly less common, recursive definition of graded reverse lexicographic order. First, we order the $\binom{n}{n-1}$ subsets of $[n]$ in lexicographic order, creating sets $S_1,\ldots, S_{n}$. Next, the sets $S_1,\ldots,S_{n}$ are iterated, and for each $S_i$, the $\binom{n-1}{n-2}$ subsets of $S_i$ are iterated in lexicographic order, etc.. This order is pictorially represented as a tree in Ex. \ref{ex_order}.
\begin{example} \label{ex_order} \emph{Here we pictorially order the set of integers $[5]$ according to $\succeq_D$.}

\vspace{-15pt}
\begin{minipage}{0.22\linewidth}
\begin{center}
\includegraphics[page=1,scale=0.20, trim=0 0 0 0]{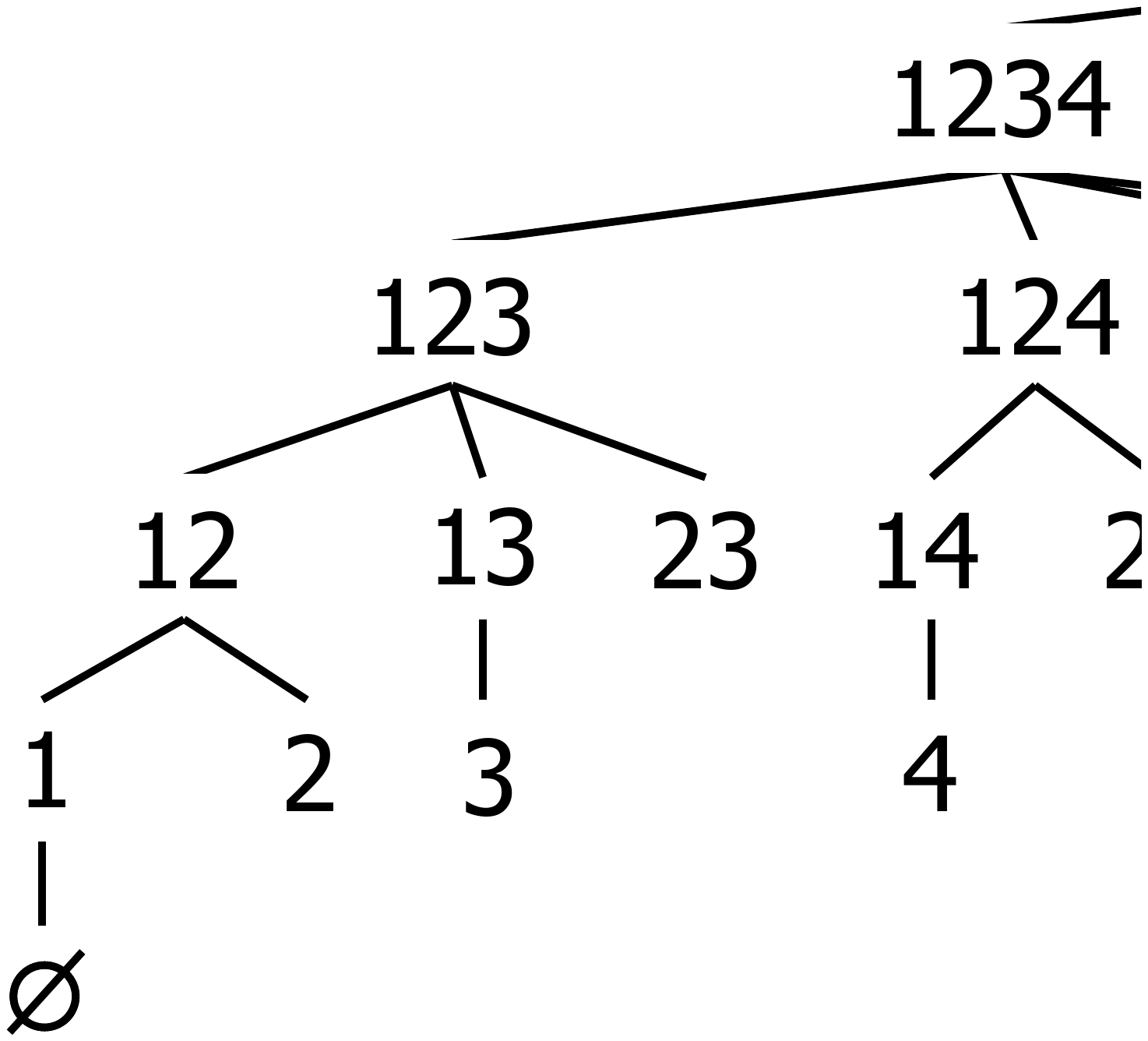}
\end{center}
\end{minipage}
\begin{minipage}{0.22\linewidth}
\begin{center}
\includegraphics[page=2,scale=0.20, trim=0 0 0 0]{order2.pdf}
\end{center}
\end{minipage}
\begin{minipage}{0.22\linewidth}
\begin{center}
\includegraphics[page=3,scale=0.20, trim=0 0 25 0]{order2.pdf}
\end{center}
\end{minipage}
\begin{minipage}{0.25\linewidth}
\begin{center}
\includegraphics[page=4,scale=0.20, trim=0 0 25 0]{order2.pdf}
\end{center}
\end{minipage}
\vspace{-45pt}

\emph{Using this tree, if two sets $S,S'$ are from different levels in the tree with $S$ higher than $S'$, then $S \succeq_D S'$. For example, $\{1245\}  \succeq_D \{234\}$. Additionally, if $S,S'$ are from the same level in the tree but $S$ appears further to the left than $S'$, then $S \succeq_D S'$. For example, $\{23\}  \succeq_D \{15\}$.
Additionally, observe that if the even and odd cardinality subsets of $[5]$ are iterated in $\succeq_D$ order, then the following pairing of even and odd subsets occurs:
\footnotesize{
\[
\begin{array}{c|c|c|c|c|c|c|c|c|c|c|c|c|c|c|c}
 {12345} & {123} & {124} & {134} & {234} & {125} & {135} & {235} & {145} & {245} & {345} & {1} & {2} & {3} & {4} & {5}\\[-2pt]
\hline
 {1234} & {1235} & {1245} & {1345} & {2345} & {12} & {13} & {23} & {14} & {24} & {34} & {15} & {25} & {35} & {45} & {\emptyset}
\end{array}
\]}\normalsize
Given a set $S$ in the pairing diagram above, if $5 \in S$, then $S$ is paired with $S \setminus 5$. If $5 \notin S$, then $S$ is paired with $S \cup 5$. This observation is proven in general in Prop. \ref{prop_pm}.\ref{prop_pm_lbl}.
}
 \hfill $\Box$
\end{example}
We refer to this tree as the \emph{order tree} of $[n]$. If two sets $S, S'$ are children of the same parent in the tree, we say that the sets are contained in the same \emph{block}. For example, $\{1,2,3\}$ and $\{2,3,4\}$ are in the same block, but $\{2,3,4\}$ and $\{1,2,5\}$ are not.

\subsection{The Partition Matrix} \label{ssec_pm} In this section, we demonstrate how to extract a $2^{n-1} \times 2^{n-1}$ matrix from the minimum-degree certificate of Thm. \ref{thm_min_cert}. We begin by considering the coefficients of $(x_i^2 - 1)$:
\begin{align*}
\Big(\sum_{ \stackrel{k \text{~even}}{k \leq n - 1}}\sum_{S \in S^{n \setminus i}_k}c_{i,S}x^S\Big)(x_i^2 - 1)~.
\end{align*}
We observe that each monomial $c_{i,S}x^S$ multiplies $(x_i^2 - 1)$, which implies that each $c_{i,S}$ appears in two equations (one corresponding to the monomial $x^Sx_i^2$, and one corresponding to the monomial $-x^S$). Thus, the unknown $c_{i,S}$ appears in the first equation with a positive coefficient, and the second equation with a negative coefficient. 
This allows us to sum the two equations, and cancel the $c$ unknowns in a cascading manner. For example, there is always one equation for the constant term:
\begin{align*}
-c_{1,\emptyset} - c_{2,\emptyset} - \cdots - c_{n,\emptyset} &= 1~.
\end{align*}
Notice that this equation sums to one, since the Nullstellensatz certificate simplifies to one. There is also always one equation for each $x_i^2$ monomial:
\begin{align}
b_iw_i + c_{i,\emptyset} &=0~. \label{eq_cb}
\end{align}
The $b_iw_i$ term appears in these equations since the product of
\begin{align*}
\Big(\sum_{ \stackrel{k \text{~odd}}{k \leq n}}\sum_{S \in S^n_k}b_Sx^S\Big)\Big(w_1x_1 + \cdots + w_nx_n\Big)~,
\end{align*}
contributes the term $b_ix_i \cdot w_ix_i = b_iw_ix_i^2$, among others.
Notice that Eq. \ref{eq_cb} sums to zero, since every monomial other than the constant term must cancel in a Nullstellensatz certificate. This set of $n+1$ equations yields the following subsystem:
\begin{align*}
-c_{1,\emptyset} - c_{2,\emptyset} - \cdots - c_{n,\emptyset} &= 1~, \quad \text{(constant term)}\\
b_1w_1 + c_{1,\emptyset} &=0~,\quad \text{($x_1^2$)}\\[-5pt]
\vdots \quad & \phantom{=} \hspace{8pt} \vdots\\[-5pt]
b_nw_n + c_{n,\emptyset} &=0~. \quad \text{($x_n^2$)}
\end{align*}
Summing these $n + 1$ equations together yields the following equation (in $b$ only): 
\begin{align*}
\sum_{i=1}^n b_iw_i &= 1~. 
\end{align*}
In general, let $S \subseteq [n]\setminus i$ be an even cardinality subset, and consider the two monomials $x^Sx_i^2$ and $x^S$. Then, the following $n - |S| + 1$ equations are always present in the extracted linear system:
\begin{align}
b_{S \cup i}w_i + c_{i,S} &= 0~, \quad \quad (x^Sx_i^2)~, \text{~~~for each $i \notin S$} \label{eq_c_alone}\\
\sum_{j \in S}b_{S \setminus j}w_j - \sum_{i \notin S}c_{i,S} &= 0~, \quad \quad (x^S)~. \nonumber
\end{align}
Summing up these $n - |S| + 1$ equations together yields the following equation (in $b$ only):
\begin{align*}
\sum_{j \notin S}b_{S \cup j}w_j  + \sum_{j \in S}b_{S \setminus j}w_j&= 0~. 
\end{align*}
\begin{defi} \label{def_pm} \emph{Given a set of integers $W = \{w_1,\ldots, w_n\}$, the coefficient matrix of the following square system of linear equations
\begin{align*}
\sum_{j \notin S}b_{S \cup j}w_j  + \sum_{j \in S}b_{S \setminus j}w_j&= 0~, \quad \text{for each $S \in \big(S^n_k \setminus \emptyset \big)$ with $|S|$ even}\\
\sum_{i=1}^n b_iw_i &= 1~,
\end{align*}
defines a $2^{n-1} \times 2^{n-1}$ matrix with columns indexed by the unknowns $b_S$ (corresponding to the $2^{n-1}$ odd cardinality subsets of $[n]$), and rows indexed by the sets $S$ (corresponding to the $2^{n-1}$ even cardinality subsets of $[n]$, including $\emptyset$). This matrix is the \emph{partition matrix}, denoted by $\Pi(W)$, with rows and columns ordered by graded reverse lexicographic order.}
\end{defi}

By studying Eq. \ref{eq_c_alone}, we see that each $c$ unknown appears in exactly one equation along with exactly one $b$ unknown. Thus, solving for the $b$ unknowns \emph{uniquely determines the entire certificate}, and determining  whether or not a given set $W$ is partitionable depends entirely on the \emph{determinant of the partition matrix}.
\begin{example} \label{ex_pm_4} \emph{
Let $W = \{w_1,w_2,w_3\}$. Via Thm. \ref{thm_min_cert}, the Nullstellensatz certificate is:
\begin{align*}
1 &= (c_{1,\emptyset} + c_{1,\{23\}}x_2x_3)(x_1^2-1) + (c_{2,\emptyset} + c_{2,\{13\}}x_1x_3)(x_2^2-1) + (c_{3,\emptyset} + c_{3,\{12\}}x_1x_2)(x_3^2-1)\\
&\phantom{=} + (b_1x_1+b_2x_2 + b_3x_3+b_{123}x_1x_2x_3)(w_1x_1+w_2x_2 + w_3x_3)~.
\end{align*}
If $W$ is not partitionable, there must exist an assignment to the unknowns $c$ and $b$ such that the certificate simplifies to one. In other words, the following system of linear equations has a solution:
\footnotesize
\[
\begin{array}{rrlrrl}
(x_1^2) & \quad c_{1,\emptyset} + b_1w_1 &=0~,&
(x_2x_3) &\quad  -c_{1,\{23\}} + b_2w_3 + b_3w_2 &=0~,
\\[-5pt]
(x_2^2)  &\quad  c_{2,\emptyset} + b_2w_2 &=0~,&
(x_1x_2x_3^2)  & \quad c_{3,\{12\}} + b_{123}w_3 &=0~,
\\[-5pt]
(x_3^2)  &\quad  c_{3,\emptyset} + b_3w_3 &=0~,&
(x_1x_2^2x_3)   &\quad c_{2,\{13\}} + b_{123}w_2 &=0~,
\\[-5pt]
(x_1x_2)  &\quad -c_{3,\{12\}}  + b_1w_2 + b_{2}w_1 &=0~,&
(x_1^2x_2x_3) & \quad c_{1,\{23\}} + b_{123}w_1 &=0~,
\\[-5pt]
(x_1x_3)  &\quad -c_{2,\{13\}} + b_1w_3 + b_3w_1 &=0~,&
(\text{constant term}) &\quad  -c_{1,\emptyset} - c_{2,\emptyset} - c_{3,\emptyset}&= 1~.
\end{array}
\] \normalsize
Following the simplifications described above, we extract a \emph{square} system of linear equations that contain only the $b$ unknowns from these equations:
\begin{align*}
b_{123}w_3  + b_1w_2 + b_{2}w_1 &=0~,  \quad  S = \{1,2\}~, &
b_{123}w_2 + b_1w_3 + b_3w_1 &=0~, \quad  S = \{1,3\}~,\\
b_{123}w_1 + b_2w_3 + b_3w_2  &=0~, \quad  S = \{2,3\}~, &
b_1w_1 + b_2w_2 + b_3w_3 &=1~. \quad S = \emptyset~.
\end{align*}
Ordering the columns as $\{b_{123}, b_1,b_2, b_3\}$, the \emph{partition matrix} is as follows:
\vspace{-10pt}
\[
\begin{array}{c}
\phantom{\big(}\{1,2\}\\[-7pt]
\phantom{\big(}\{1,3\}\\[-7pt]
\phantom{\big(}\{2,3\}\\[-7pt]
\phantom{\big(}\emptyset
\end{array}
\stackrel{
\begin{array}{cccc}
\hspace{-5pt}b_{123} & \hspace{2pt}b_1 & \hspace{2pt}b_2 & \hspace{4pt}b_3
\end{array}
}{
\left[ \begin{array}{cccc}
w_3 & w_2 & w_1 & 0\\[-7pt]
w_2 & w_3 & 0   & w_1 \\[-7pt]
w_1 & 0   & w_3 & w_2 \\[-7pt]
0   & w_1 & w_2 & w_3
\end{array} \right]}
\]
As a preview of our main result, we note that the determinant of this matrix is
\begin{align*} (w_1 + w_2 + w_3)(-w_1 + w_2 + w_3)(w_1 - w_2 + w_3)(-w_1 - w_2 + w_3)~,
\end{align*}
which represents a brute-force iteration over all of the possible partitions of the set $W$. This will be formally defined as the \emph{partition polynomial} in Sec. \ref{sec_det}.} \hfill $\Box$\end{example}

For the duration of this section, we collect a few essential facts about the partition matrix $\Pi(W)$. We also provide a slightly larger example of the partition matrix to demonstrate these properties.

\begin{example} \label{ex_pm_5_ord} \emph{
Given $W = \{w_1,\ldots,w_5\}$, here is the $16 \times 16$ partition matrix $\Pi(W)$.
\small{\[
\begin{array}{c|cccccccccccccccc}
 & {12345} & {123} & {124} & {134} & {234} & {125} & {135} & {235} & {145} & {245} & {345} & {1} & {2} & {3} & {4} & {5}\\
\hline \hline
\vspace{-5pt}
{1234} & w_5 & w_4 & w_3 & w_2 & w_1 & 0 & 0 & 0 & 0 & 0 & 0 & 0 & 0 & 0 & 0 & 0\\ 
\vspace{-5pt}
{1235} & w_4 & w_5 & 0 & 0 & 0 & w_3 & w_2 & w_1 & 0 & 0 & 0 & 0 & 0 & 0 & 0 & 0\\ 
\vspace{-5pt}
{1245} & w_3 & 0 & w_5 & 0 & 0 & w_4 & 0 & 0 & w_2 & w_1 & 0 & 0 & 0 & 0 & 0 & 0\\ 
\vspace{-5pt}
{1345} & w_2 & 0 & 0 & w_5 & 0 & 0 & w_4 & 0 & w_3 & 0 & w_1 & 0 & 0 & 0 & 0 & 0\\ 
\vspace{-5pt}
{2345} & w_1 & 0 & 0 & 0 & w_5 & 0 & 0 & w_4 & 0 & w_3 & w_2 & 0 & 0 & 0 & 0 & 0\\ 
\vspace{-5pt}
{12} & 0 & w_3 & w_4 & 0 & 0 & w_5 & 0 & 0 & 0 & 0 & 0 & w_2 & w_1 & 0 & 0 & 0\\ 
\vspace{-5pt}
{13} & 0 & w_2 & 0 & w_4 & 0 & 0 & w_5 & 0 & 0 & 0 & 0 & w_3 & 0 & w_1 & 0 & 0\\ 
\vspace{-5pt}
{23} & 0 & w_1 & 0 & 0 & w_4 & 0 & 0 & w_5 & 0 & 0 & 0 & 0 & w_3 & w_2 & 0 & 0\\ 
\vspace{-5pt}
{14} & 0 & 0 & w_2 & w_3 & 0 & 0 & 0 & 0 & w_5 & 0 & 0 & w_4 & 0 & 0 & w_1 & 0\\ 
\vspace{-5pt}
{24} & 0 & 0 & w_1 & 0 & w_3 & 0 & 0 & 0 & 0 & w_5 & 0 & 0 & w_4 & 0 & w_2 & 0\\ 
\vspace{-5pt}
{34} & 0 & 0 & 0 & w_1 & w_2 & 0 & 0 & 0 & 0 & 0 & w_5 & 0 & 0 & w_4 & w_3 & 0\\ 
\vspace{-5pt}
{15} & 0 & 0 & 0 & 0 & 0 & w_2 & w_3 & 0 & w_4 & 0 & 0 & w_5 & 0 & 0 & 0 & w_1\\ 
\vspace{-5pt}
{25} & 0 & 0 & 0 & 0 & 0 & w_1 & 0 & w_3 & 0 & w_4 & 0 & 0 & w_5 & 0 & 0 & w_2\\ 
\vspace{-5pt}
{35} & 0 & 0 & 0 & 0 & 0 & 0 & w_1 & w_2 & 0 & 0 & w_4 & 0 & 0 & w_5 & 0 & w_3\\ 
\vspace{-5pt}
{45} & 0 & 0 & 0 & 0 & 0 & 0 & 0 & 0 & w_1 & w_2 & w_3 & 0 & 0 & 0 & w_5 & w_4\\ 
\vspace{-5pt}
{\emptyset} & 0 & 0 & 0 & 0 & 0 & 0 & 0 & 0 & 0 & 0 & 0 & w_1 & w_2 & w_3 & w_4 & w_5 
\end{array}
\]   \vspace{-15pt} \hfill $\Box$ }\normalsize} 
\end{example}

\vspace{5pt}
\begin{prop} \label{prop_pm} Given $W = \{w_1,\ldots,w_n\}$, the $2^{n-1} \times 2^{n-1}$ partition matrix $\Pi(W)$ has the following properties:
\begin{enumerate}
	\item \label{prop_pm_once} The entry $w_i$ with $i = \{1,\ldots,n\}$  appears exactly once in each row and column.
	\item \label{prop_pm_lbl} If row $i$ is indexed by set $S \subseteq [n]$ (with $|S|$ even), and $n \in S$, then column $i$ is indexed by $S \setminus n$. If $n \notin S$, then column $i$ is indexed by $S \cup n$.
	\item \label{prop_pm_diag} All diagonal entries of $\Pi(W)$ are equal to $w_n$.
	\item \label{prop_pm_sym} $\Pi(W)$ is symmetric.
\end{enumerate}
\end{prop}
\noindent \textit{Proof of Prop. \ref{prop_pm}.\ref{prop_pm_once}}: After inspecting the equation defining the partition matrix
\begin{align*}
\sum_{j \notin S}b_{S \cup j}w_j  + \sum_{j \in S}b_{S \setminus j}w_j&= 0~,
\end{align*}
where $S$ represents an even cardinality subset of $[n]$, 
it is evident that each row contains exactly one entry for each $w_i$. To see that each column also contains exactly one entry $w_i$ with $i = \{1,\ldots,n\}$, consider the column indexed by unknown $b_{S'}$ where $S' \subseteq [n]$ with odd cardinality. Then, for each $j \in S'$, the row indexed by $S = S' \setminus j$ contains $w_j$. Additionally, for each $j \notin S'$, the row indexed by $S = S' \cup j$ also contains $w_j$. Thus, each row and column contains exactly one entry $w_i$ for $i = \{1,\ldots,n\}$. \hfill $\Box$	
\vspace{10pt}

As an example of Prop. \ref{prop_pm}.\ref{prop_pm_lbl}, note that row $\{12\}$ is paired with column $\{125\}$, and column $\{1\}$ is paired with row $\{15\}$ in Ex. \ref{ex_pm_5_ord}.

\noindent \textit{Proof of Prop. \ref{prop_pm}.\ref{prop_pm_lbl}}: To prove this claim, suppose that we have the ``order tree" $T_{n-1}$ for the subsets of $[n-1]$. In order to create the order tree $T_n$ for the subsets of $[n]$, we first copy $T_{n-1}$ and add the integer $n$ to each set, creating the tree $T_{n-1}\cup n$. We then join the node in $T_{n-1}\cup n$ indexed by $\{1\cdots n\}$ to the node in $T_{n-1}$ indexed $\{1\cdots (n-1)\}$. The resulting tree is the order tree for $T_n$. For example,

\vspace{-12pt}
\hspace{35pt}\begin{minipage}{0.21\linewidth}
\begin{center}
\includegraphics[page=1,scale=0.19, trim=0 0 25 0]{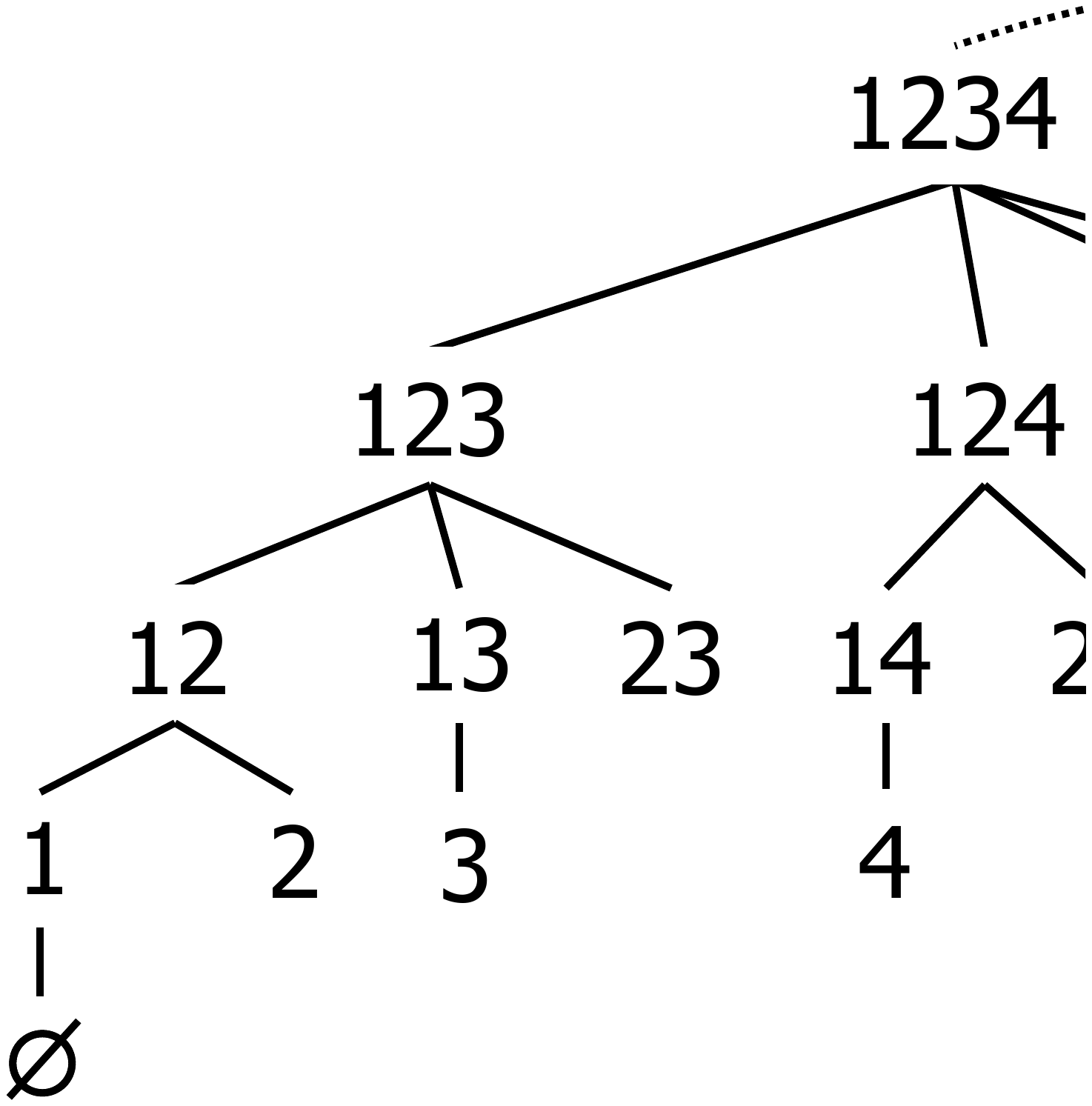}
\end{center}
\end{minipage}
\begin{minipage}{0.205\linewidth}
\begin{center}
\includegraphics[page=2,scale=0.19, trim=0 0 25 0]{order_contract2.pdf}
\end{center}
\end{minipage}
\begin{minipage}{0.20\linewidth}
\begin{center}
\includegraphics[page=3,scale=0.19, trim=0 0 25 0]{order_contract2.pdf}
\end{center}
\end{minipage}
\begin{minipage}{0.28\linewidth}
\begin{center}
\includegraphics[page=4,scale=0.19, trim=0 0 25 0]{order_contract2.pdf}
\end{center}
\end{minipage}
\vspace{-60pt}

Since no set in $T_{n-1}$ contains the integer $n$ and every set in $T_{n-1} \cup n$ contains $n$, it is easy to see that the even and odd sets are paired by inspecting how $T_{n-1}$ overlays on top of $T_{n-1}\cup n$. Thus, the claim holds. \hfill $\Box$

\noindent \textit{Proof of Prop. \ref{prop_pm}.\ref{prop_pm_diag}}: This result follows from the equations defining the partition matrix, and also Prop. \ref{prop_pm}.\ref{prop_pm_lbl}, which defines the row-column pairing of the diagonal element. \hfill $\Box$


\noindent \textit{Proof of Prop. \ref{prop_pm}.\ref{prop_pm_sym}}: Consider an arbitrary row $i$ indexed by a set $S_i$, and let column $i$ be indexed by the set $b_i$. In order to prove symmetry, we must show that row $i$ is equal to column $i$. By Prop. \ref{prop_pm}.\ref{prop_pm_lbl}, $n$ is either in $S_i$ or $b_i$, but not both. Without loss of generality, assume $n \in S_i$ and $b_i = S_i \setminus n$ (e.g. $S_i = \{15\}$ and $b_i = \{1\}$).  Suppose $\big(\Pi(W)\big)_{ij} = w_{k_j}$ for $j < i$. We must show that $\big(\Pi(W)\big)_{ji} = w_{k_j}$. Since $j < i$, $k_j \notin S_i$, and column $j$ is indexed by $b_j = S_i \cup k_j$ (e.g. in row $\{15\}$, $w_2$ appears in column $\{125\}$). Since $n \in (S_i \cup k_j)$, row $j$ is indexed by $S_j = (S_i \cup k_j) \setminus n$ (e.g. $S_j = \{12\}$). Then, $S_j \setminus k_j = b_i$, and $\big(\Pi(W)\big)_{ji}= w_{k_j}$. 

Suppose $i < j$, and $\big(\Pi(W)\big)_{ij}$ is again equal to $w_{k_j}$. Then $k_j \in S_i$, and column $j$ is indexed by $b_j = S_i \setminus k_j$ and row $j$ is indexed by $S_j = (S_i \setminus k_j) \setminus n$ (e.g., in row $\{15\}$, $w_1$ appears in column $\{1\}$). But then $S_j \cup k_j = b_i$, and $\big(\Pi(W)\big)_{ji} = w_{k_j}$.

A similar argument holds if $n \notin S_i$, but with the logic reversed. Since we have shown that $\big(\Pi(W)\big)_{ij} =\big(\Pi(W)\big)_{ji}$. we have shown that the matrix is symmetric. \hfill $\Box$

\subsection{The Partition Matrix as a Sum of Permutation Matrices} \label{ssec_pm_pi} We will now express the partition matrix in terms of a specific set of permutation matrices, and then prove a series of properties about these particular permutation matrices. Recall that the \emph{symmetric difference} of two sets $A$ and $B$ (denoted as $A \Delta B$) is the set of elements which are in either set $A$ or $B$ but not in their intersection. For example, the symmetric difference $\{1,2,3\} \Delta \{3,4\} = \{1,2,4\}$. By Prop. \ref{prop_pm}.\ref{prop_pm_sym}, every $w_k$ appears exactly once in each row and column. Therefore, we can express $\Pi(W)$  as follows.
\begin{defi} \label{def_pi_k} Given $W = \{w_1,\ldots,w_n\}$ and the corresponding partition matrix $\Pi(W)$, let $\Pi_1,\ldots, \Pi_n \in \{0,1\}^{2^{n-1} \times 2^{n-1}}$ be permutation matrices such that
\begin{align*}
(\Pi_k)_{ij} &= \begin{cases}
1 & \text{if $\big(\Pi(W)\big)_{ij} = w_k$}~,\\
0 & \text{otherwise}~.
\end{cases}~, \quad \quad \text{where $1 \leq i,j \leq 2^{n-1}$.}
\end{align*}
\end{defi}
By this definition, it is clear that
\begin{align*}
\Pi(W) &= \sum_{k=1}^n w_k \Pi_k~.
\end{align*}

\begin{example} \label{ex_pm_pi_3} \emph{
Let $W = \{w_1,\ldots,w_4\}$. Here we display the $8 \times 8$ partition matrix $\Pi(W)$, and the particular permutation matrix $\Pi_3$. For convenience, we highlight the $w_3$ entries appearing in $\Pi(W)$.
\small{\[
\Pi(W) = \begin{array}{c|cccccccc}
 & {123} & {124} & {134} & {234} & {1} & {2} & {3} & {4}\\
\hline \hline
\vspace{-5pt}
{1234} & w_4 & \mathbf{w_3} & w_2 & w_1 & 0 & 0 & 0 & 0\\ 
\vspace{-5pt}
{12} & \mathbf{w_3} & w_4 & 0 & 0 & w_2 & w_1 & 0 & 0\\
\vspace{-5pt}
{13} & w_2 & 0 & w_4 & 0 &\mathbf{w_3} & 0 & w_1 & 0\\
\vspace{-5pt}
{23} & w_1 & 0 & 0 & w_4 & 0 & \mathbf{w_3} & w_2 & 0\\
\vspace{-5pt}
{14} & 0 & w_2 & \mathbf{w_3} & 0 & w_4 & 0 & 0 & w_1\\
\vspace{-5pt}
{24} & 0 & w_1 & 0 & \mathbf{w_3} & 0 & w_4 & 0 & w_2\\
\vspace{-5pt}
{34} & 0 & 0 & w_1 & w_2 & 0 & 0 & w_4 & \mathbf{w_3}\\
\vspace{-5pt}
{\emptyset} & 0 & 0 & 0 & 0 & w_1 & w_2 & \mathbf{w_3} & w_4
\end{array}~, \quad \quad \Pi_{\mathbf{3}} = \begin{array}{c|cccccccc}
 & {123} & {124} & {134} & {234} & {1} & {2} & {3} & {4}\\
\hline \hline
\vspace{-5pt}
{1234} & 0 &\mathbf{1} &0 & 0 & 0 & 0 & 0 & 0\\ 
\vspace{-5pt}
{12} & \mathbf{1} & 0 & 0 & 0 & 0 & 0 & 0 & 0\\
\vspace{-5pt}
{13} & 0 & 0 & 0 & 0 & \mathbf{1} & 0 & 0 & 0\\
\vspace{-5pt}
{23} & 0 & 0 & 0 &0 & 0 &\mathbf{1} & 0 & 0\\
\vspace{-5pt}
{14} & 0 & 0 &\mathbf{1} & 0 &0 & 0 & 0 &0\\
\vspace{-5pt}
{24} & 0 & 0 & 0 & \mathbf{1} & 0 &0 & 0 &0\\
\vspace{-5pt}
{34} & 0 & 0 &0 &0 & 0 & 0 &0 &\mathbf{1}\\
\vspace{-5pt}
{\emptyset} & 0 & 0 & 0 & 0 &0 & 0 &\mathbf{1} & 0
\end{array}\]   \vspace{-15pt} \hfill $\Box$ }\normalsize} 
\end{example}
\vspace{5pt}
By Prop. \ref{prop_pm}.\ref{prop_pm_sym}, the matrix $\Pi(W)$ is symmetric. Therefore, each of the permutation matrices $\Pi_1,\ldots,\Pi_n$ is likewise symmetric, and the following proposition holds. 
\begin{prop} \label{prop_pi} Given $W = \{w_1,\ldots,w_n\}$ and the corresponding permutation matrices $\Pi_1,\ldots, \Pi_n$, each of the following holds:
\begin{enumerate}
\item \label{prop_pi_n} $\Pi_n$ is the identity matrix~,
\item \label{prop_pi_inv} For $k=1,\ldots,n$, $\Pi_k^2=I$ (the matrices are \emph{involutory})~,
\item \label{prop_pi_eigen} For $k=1,\ldots,n-1$, $\Pi_k$ has $\pm 1$ eigenvalues, and $\Pi_n$ has all $+1$ eigenvalues.
\item \label{prop_pi_diag} For $k=1,\ldots,n$, $\Pi_k$ is diagonalizable, and
\item \label{prop_pi_com} $\Pi_k\Pi_l = \Pi_l\Pi_k$ (the permutation matrices $\Pi_k$ pairwise commute)~.
\end{enumerate}
\end{prop}
\begin{example} \label{ex_pm_5_ord_com} \emph{
Given $W = \{w_1,\ldots,w_5\}$, here is the $16 \times 16$ partition matrix $\Pi(W)$. We will show that $\big(\Pi_1\Pi_3\big)_{(\{23\}, \{125\})} = 1 = \big(\Pi_3\Pi_1\big)_{(\{23\}, \{125\})}$~.
\small{\[
\begin{array}{c|cccccccccccccccc}
 & {12345} & {123} & {124} & {134} & {234} & {125} & {135} & {235} & {145} & {245} & {345} & {1} & {2} & {3} & {4} & {5}\\
\hline \hline
\vspace{-5pt}
{1234} & w_5 & w_4 & w_3 & w_2 & w_1 & 0 & 0 & 0 & 0 & 0 & 0 & 0 & 0 & 0 & 0 & 0\\ 
\vspace{-5pt}
{1235} & w_4 & w_5 & 0 & 0 & 0 & \mathbf{w_3} & w_2 & w_1 & 0 & 0 & 0 & 0 & 0 & 0 & 0 & 0\\ 
\vspace{-5pt}
{1245} & w_3 & 0 & w_5 & 0 & 0 & w_4 & 0 & 0 & w_2 & w_1 & 0 & 0 & 0 & 0 & 0 & 0\\ 
\vspace{-5pt}
{1345} & w_2 & 0 & 0 & w_5 & 0 & 0 & w_4 & 0 & w_3 & 0 & w_1 & 0 & 0 & 0 & 0 & 0\\ 
\vspace{-5pt}
{2345} & w_1 & 0 & 0 & 0 & w_5 & 0 & 0 & w_4 & 0 & w_3 & w_2 & 0 & 0 & 0 & 0 & 0\\ 
\vspace{-5pt}
{12} & 0 & w_3 & w_4 & 0 & 0 & w_5 & 0 & 0 & 0 & 0 & 0 & w_2 & w_1 & 0 & 0 & 0\\ 
\vspace{-5pt}
{13} & 0 & w_2 & 0 & w_4 & 0 & 0 & w_5 & 0 & 0 & 0 & 0 & w_3 & 0 & w_1 & 0 & 0\\ 
\vspace{-5pt}
{23} & 0 & \mathbf{w_1} & 0 & 0 & w_4 & 0 & 0 & w_5 & 0 & 0 & 0 & 0 & \mathbf{w_3} & w_2 & 0 & 0\\ 
\vspace{-5pt}
{14} & 0 & 0 & w_2 & w_3 & 0 & 0 & 0 & 0 & w_5 & 0 & 0 & w_4 & 0 & 0 & w_1 & 0\\ 
\vspace{-5pt}
{24} & 0 & 0 & w_1 & 0 & w_3 & 0 & 0 & 0 & 0 & w_5 & 0 & 0 & w_4 & 0 & w_2 & 0\\ 
\vspace{-5pt}
{34} & 0 & 0 & 0 & w_1 & w_2 & 0 & 0 & 0 & 0 & 0 & w_5 & 0 & 0 & w_4 & w_3 & 0\\ 
\vspace{-5pt}
{15} & 0 & 0 & 0 & 0 & 0 & w_2 & w_3 & 0 & w_4 & 0 & 0 & w_5 & 0 & 0 & 0 & w_1\\ 
\vspace{-5pt}
{25} & 0 & 0 & 0 & 0 & 0 & \mathbf{w_1} & 0 & w_3 & 0 & w_4 & 0 & 0 & w_5 & 0 & 0 & w_2\\ 
\vspace{-5pt}
{35} & 0 & 0 & 0 & 0 & 0 & 0 & w_1 & w_2 & 0 & 0 & w_4 & 0 & 0 & w_5 & 0 & w_3\\ 
\vspace{-5pt}
{45} & 0 & 0 & 0 & 0 & 0 & 0 & 0 & 0 & w_1 & w_2 & w_3 & 0 & 0 & 0 & w_5 & w_4\\ 
\vspace{-5pt}
{\emptyset} & 0 & 0 & 0 & 0 & 0 & 0 & 0 & 0 & 0 & 0 & 0 & w_1 & w_2 & w_3 & w_4 & w_5 
\end{array}
\]}\normalsize} 
Observe that $\Pi_1\big(\{23\}, \{123\}\big) = 1 = \Pi_3\big(\{1235\}, \{125\}\big)$. Furthermore, observe that $\big(\{1235\}, \{123\}\big)$ indexes a diagonal element since $\{1235\} \Delta \{5\} = \{123\}$ (by Prop. \ref{prop_pm}.\ref{prop_pm_lbl}). For the commuted multiplication $\Pi_3\Pi_1$, observe that $\Pi_3\big(\{23\}, \{2\}\big) = 1 = \Pi_1\big(\{25\}, \{125\}
\big)$. Furthermore, observe that $\big(\{25\}, \{2\}\big)$ indexes a diagonal element since $\{25\} \Delta \{5\} = \{2\}$. This is the technique used to prove Prop. \ref{prop_pi}.\ref{prop_pi_com}. $\Box$
\end{example}
\begin{proof} Prop. \ref{prop_pi}.\ref{prop_pi_n} comes directly from the definition of the permutation matrices and Prop. \ref{prop_pm}.\ref{prop_pm_diag}. To see that the matrices are involutory (Prop. \ref{prop_pi}.\ref{prop_pi_inv}), recall that for any permutation matrix $P$, $PP^T = I$. In this case, since the matrices are symmetric, $\Pi_k = \Pi_k^T$, and thus $\Pi_k^2=I$ follows. For Prop. \ref{prop_pi}.\ref{prop_pi_eigen}, since the matrices $\Pi_k$ are involutory, the minimal polynomial is $x^2 - 1$, and because the eigenvalues are roots of the minimal polynomial, the eigenvalues are clearly $\pm 1$.  Since $\Pi_n = I$, the eigenvalues of $\Pi_n$ all clearly $+1$. For Prop. \ref{prop_pi}.\ref{prop_pi_diag}, we observe that this is proven in \cite{hoffman_kunze}, Theorem 6, pg. 204.

Finally, we must show that the matrices pairwise commute. We will show that, given $k_1,k_2 \in \{1,\ldots,n\}$,  $\Pi_{k_1}\Pi_{k_2} = \Pi_{k_2}\Pi_{k_1}$. Since every row/column of $\Pi_k$ has exactly one non-zero entry, $\big(\Pi_{k_1}\Pi_{k_2}\big)_{ij} = 1$ only when the non-zero entries are located in the same row/column, respectively.

In particular, let $S_i$ index the $i$-th row of $\Pi(W)$, and let $S_j$ index the $j$-th column of $\Pi(W)$. For shorthand, we will denote $\big(\Pi_{k} \big)_{S_i,S_j}$ as $\big(\Pi_{k} \big)_{ij}$ with $1 \leq i,j \leq 2^{n-1}$. Then, the non-zero entries in row $S_i$ of $\Pi_{k_1}$ and column $S_j$ in $\Pi_{k_2}$ can be expressed as $\Pi_{k_1}
\big(S_i, S_i \Delta k_1 \big) = 1= \Pi_{k_2}\big(S_j \Delta k_2, S_j \big)$ (relevant to the product $\Pi_{k_1}\Pi_{k_2}$), and the non-zero entries in row $S_i$ of $\Pi_{k_2}$ and column $S_j$ in $\Pi_{k_1}$ can be expressed as $\Pi_{k_2}
\big(S_i, S_i \Delta k_2 \big) = 1 = \Pi_{k_1}\big(S_j \Delta k_1, S_j \big)$ (relevant to the commuted product $\Pi_{k_2}\Pi_{k_1}$). This can be seen by recalling the definitions of both the partition matrix and permutation matrices. We observe that $\big(\Pi_{k_1}\Pi_{k_2}\big)_{ij} = 1$ if and only if $\{S_j \Delta k_2, S_i \Delta k_1\}$ indexes a diagonal entry, and $\big(\Pi_{k_2}\Pi_{k_1}\big)_{ij} = 1$ if and only if $\{S_j \Delta k_1, S_i \Delta k_2\}$ indexes a diagonal entry. Therefore, in order to show $\big(\Pi_{k_1}\Pi_{k_2}\big)_{ij} = \big(\Pi_{k_2}\Pi_{k_1}\big)_{ij}$, we simply observe that if $\{S_j \Delta k_2, S_i \Delta k_1\}$ indexes a diagonal entry, then $S_j \Delta k_2 \Delta n =  S_i \Delta k_1$ (by Prop. \ref{prop_pm}.\ref{prop_pm_lbl}). However, if $S_j \Delta k_2 \Delta n =  S_i \Delta k_1$, then $S_j \Delta k_1 \Delta n = S_i \Delta k_2$, by the definition of the symmetric difference. Therefore, $\big(\Pi_{k_1}\Pi_{k_2}\big)_{ij} = \big(\Pi_{k_2}\Pi_{k_1}\big)_{ij}$, and the matrices pairwise commute.
\end{proof}

We pause to observe that the set of matrices $\{\Pi_1,\ldots,\Pi_n\}$ has now been shown to be a set of commuting, diagonalizable matrices. Recall that a set of matrices is \emph{simultaneously diagonalizable} if there exists a single invertible matrix $P$ such that $P^{-1}AP$ is a diagonal matrix for every $A$ in the set. This allows us to recall the following well-known fact:
\begin{prop} [\cite{horn_john}, pg. 64] \label{prop_sim_diag} A set (possibly infinite) of diagonalizable matrices is commuting if and only if it is simultaneously diagonalizable. 
%
\end{prop}

Having gathered together a series of facts about the partition matrix and the associated permutation matrices, we now investigate the determinant of the partition matrix.

\section{\large{The Partition Matrix and Partition Polynomial}} \label{sec_det} Given a square non-singular matrix $A$, Cramer's rule states that $Ax = b$ is solved by
\begin{align*}
x_i &= \frac{\det(A|^i_b)}{\det(A)}~,
\end{align*}
where $A|^i_b$ is the matrix $A$ with the $i$-th column replaced with the right-hand side vector $b$.  In Section \ref{sec_pm}, we extracted a $2^{n-1} \times 2^{n-1}$ square linear system from the general linear system constructed via the minimum-degree Nullstellensatz certificate described by Thm. \ref{thm_min_cert}. Here, we see by Cramer's rule that the unknowns within that certificate are ratios of two determinants. In this section, we show that the determinant of the partition matrix is equivalent to a brute-force iteration over all the partitions of $W$. Therefore, the denominator of any unknown in the certificate is a combinatorial representation of the partition problem.

We observe that, in general, the linear system $Ax=b$ may have a solution even if $\det(A)=0$. However, in the case of the partition matrix, when we demonstrate that the $\det(A)$ is equal to the \emph{partition polynomial}, we will be demonstrating that $Ax=b$ only has a solution in the case when $\det(A)\neq 0$. 

Let $\{-1,1\}^n$ be the set of all $\pm 1$ bit strings of length $n$. For $S \in \{-1,1\}^n$, let $s_i$ denote the $i$-th bit in the string $S$.
\begin{defi} \label{defi_part_poly} \emph{Given a set $W = \{w_1,\ldots,w_n\}$, let
\begin{align*}
\prod_{S \in \{-1,1\}^{n -1}} \Bigg( \bigg(\sum_{i = 1}^{n-1} {s_i}w_i \bigg) + w_n\Bigg)
\end{align*}
be the \emph{partition polynomial} of $W$.}
\end{defi}

For example, let $n=5$, and $S \in \{-1,1\}^4$ be $S = ``\text{-1,1,-1,-1}"$. Then, $S$ corresponds to the $-w_1 +w_2 -w_3 - w_4$, and denotes a partition of $W = \{w_1,\ldots,w_5\}$, with $w_5$ fixed on the ``positive" side of the partition, and the other $w_i$ sorted according to sign.
\vspace{-7pt}
\[\small{
\begin{array}{c|c}
- & +\\[-3pt] \hline
w_1 & w_5\\[-5pt]
w_3 & w_2\\[-5pt]
w_4
\end{array}}\normalsize
\]
If this arrangement of $w_i$ is a partition of $W$, then $-w_1 +w_2 -w_3 - w_4\mathbf{ + w_5} = 0$. In this way, \emph{any} bitstring $S \in \{-1,1\}^{n-1}$ is equivalent to fixing $w_n$ on the ``positive" side of the partition, and then arranging the other $w_i$ on the ``positive/negative" side, according to sign. In this way, the partition polynomial represents an iteration over \emph{every possible partition of $W$}, avoiding double-counting by permanently fixing $w_n$ on the ``positive" side. If the set $W$ is partitionable, one of bitstrings $S$ will define a factor of the partition polynomial that sums to zero. We will show that the determinant of the partition matrix is the partition polynomial: therefore, if the determinant of the partition matrix is zero, the linear system has \emph{no} solution, and there is \emph{no} Nullstellensatz certificate. 
\begin{example} \label{ex_b4} \emph{ In Ex. \ref{ex_cert}, we presented an actual minimum-degree certificate for the non-partitionable set $W = \{1, 3, 5, 2\}$~. We observe that 
\begin{align*}
-51975 &= (1+3+5+2)(-1+3+5+2)(1-3+5+2)(1+3-5+2)\\
&\phantom{=}(-1-3+5+2)(-1+3-5+2)(1-3-5+2)(-1-3-5+2)~.
\end{align*}
Via Cramer's rule, we see that the unknown $b_4$ is equal to
\begin{align*}
b_4 &= \frac{-2550}{-51975} = \frac{34}{693}~,
\end{align*}
which is indeed the value of unknown $b_4$ as it appears in the certificate.} \hfill $\Box$
\end{example}
\begin{example} \label{ex_part_poly} \emph{Here is the determinant of the $8 \times 8$ partition matrix $\text{Part}_4$:}

\begin{minipage}{0.5\linewidth}
\[
\det\left(\left[ \begin{array}{cccccccc}
\vspace{-5pt}
w_4 & w_3 & w_2 & w_1 & 0 & 0 & 0 & 0\\ 
\vspace{-5pt}
w_3 & w_4 & 0 & 0 & w_2 & w_1 & 0 & 0\\
\vspace{-5pt}
w_2 & 0 & w_4 & 0 & w_3 & 0 & w_1 & 0\\
\vspace{-5pt}
w_1 & 0 & 0 & w_4 & 0 & w_3 & w_2 & 0\\
\vspace{-5pt}
0 & w_2 & w_3 & 0 & w_4 & 0 & 0 & w_1\\
\vspace{-5pt}
0 & w_1 & 0 & w_3 & 0 & w_4 & 0 & w_2\\
\vspace{-5pt}
0 & 0 & w_1 & w_2 & 0 & 0 & w_4 & w_3\\
\vspace{-5pt}
0 & 0 & 0 & 0 & w_1 & w_2 & w_3 & w_4\\
\vspace{-10pt}
\phantom{=}
\end{array} \right]\right)=~
\]
\end{minipage}
\hspace{-17pt}\begin{minipage}{0.5\linewidth}
\footnotesize{
\begin{align*}
&\phantom{=}(w_1+w_2+w_3+w_4)(-w_1+w_2+w_3+w_4)\\
&\phantom{=}(w_1-w_2+w_3+w_4)(w_1+w_2-w_3+w_4)\\
&\phantom{=}(-w_1+w_2-w_3+w_4)(-w_1-w_2+w_3+w_4)\\
&\phantom{=}(w_1-w_2-w_3+w_4)(-w_1-w_2-w_3+w_4)~.
\end{align*}}\normalsize
\end{minipage}
\end{example}
\begin{thm} \label{thm_pm_det} Given $W = \{w_1,\ldots,w_n\}$, the determinant of the partition matrix of $W$ is the partition polynomial of $W$. 
\end{thm}
\begin{proof}
Since $\Pi_1,\ldots, \Pi_n$ are a set of pairwise commuting diagonalizable matrices, via Prop. \ref{prop_sim_diag}, they are also simultaneously diagonalizable. Let $P$ be the $2^{n-1} \times 2^{n-1}$ matrix that simultaneously diagonalizes $\Pi_1,\ldots, \Pi_n$. Then
\begin{align*}
P^{-1}\Pi(W)P &= P^{-1}\Bigg( \sum_{k = 1}^n w_k \Pi_k \Bigg) P = \sum_{k = 1}^n w_k P^{-1} \Pi_k P\\ 
&=  \left[\begin{array}{cccc}
 \sum_{k=1}^n w_k \lambda_{k, 1} & & &\\
& \sum_{k=1}^n w_k \lambda_{k, 2} & &\\
&&\ddots &\\
&&& \sum_{k=1}^n w_k \lambda_{k, 2^{n-1}}\end{array} \right]
\end{align*}
where $\lambda_{k,1},\ldots, \lambda_{k,2^{n-1}} \in \{ -1, 1\}$ are the eigenvalues of $\Pi_k$. Then, we see
\begin{align*}
\det \big(\Pi(W)\big) &= \det\big(P^{-1}\big) \det\big(\Pi(W)\big) \det(P) = \det\big(P^{-1} \Pi(W) P\big)\\
&= \prod_{j=1}^{2^{n-1}}\bigg(\sum_{k=1}^n w_k \lambda_{k,j} \bigg)~.
\end{align*}
Therefore, we see that $\det \big(\Pi(W)\big)$ is a product of linear polynomials in $w_k$ with coefficients $\pm 1$, where the coefficient of $w_n$ is always $+1$.

In order to complete the proof, we must now show that \emph{every} linear polynomial of this form is present in the determinant (for example, for $n = 3$, the determinant is \emph{not} $(w_1 - w_2 + w_3)^4$). Thus, consider any linear polynomial $\sum_{k = 1}^n p_k w_k$ with $p_n = 1$ and $ p_k = \pm 1, k = 1, \ldots, n -1$, and assume this polynomial is \emph{not} a factor of $\det \big(\Pi(W)\big)$.

In order to derive a contradiction, set $w_k = -p_k$ for $k = 1,\ldots, n-1$, and $w_n = n-1$. Observe that $W = \{-p_1, -p_2,\ldots, -p_{n-1}, n-1\}$ is partitionable, according to the sign pattern of the $p_k$:
\begin{align*}
\sum_{p_k = +1} w_k - \sum_{p_k = - 1} w_k = \sum_{k = 1}^n p_k w_k = 0~.
\end{align*}
However, for every other assignment of $\lambda_n = 1, \lambda_k = \pm1$, $\sum_{k=1}^n \lambda_k w_k \geq 1$. Since  $\sum_{k = 1}^n p_k w_k$ is \emph{not} a factor of  $\det \big(\Pi(W)\big)$, then $\det \big(\Pi (W)\big) \neq 0$ for this $W$. However, we can now construct a Nullstellensatz certificate of non-partitionability, even though $W$ is partitionable, which is a contradiction. Since the linear factor was chosen at random, each of the $2^{n-1}$ linear polynomials $\sum_{k=1}^n p_k w_k, p_k = \pm 1, p_n = 1$ must appear as a factor of the determinant and
\begin{align*}
\det\big(\Pi(W)\big) &= \prod_{P \in \{-1,1\}^{n-1}} \Bigg( \bigg( \sum_{k=1}^{n-1} p_k w_k \bigg)+ w_n \Bigg) 
= \text{the partition polynomial}~. \qedhere
\end{align*}
\end{proof}

\begin{rem}
A short proof of Thm. \ref{thm_pm_det} may also be derived using
representation theory of finite groups, essentially, 
applying \cite[Thm. 2]{MR1659232}, as follows.
The $n-1$ non-identity permutation matrices $\Pi_1$,\dots, $\Pi_{n-1}$ giving $\Pi(W)$
in Definition~\ref{def_pi_k} generate 
an  elementary abelian 2-group $E_{2^{n-1}}$ of order $2^{n-1}$.
It can be shown that  $E_{2^{n-1}}$ acts fixed-point-free, i.e.
we have its regular representation.
Thus,  after the simultaneous diagonalization of the $\Pi_j$'s, the 
diagonal entries of the transformed $\Pi(W)$ will be in 1-to-1 correspondence with the
irreducible representations of   $E_{2^{n-1}}$, which are encoded by the $\pm 1$-signs
assigned to $\Pi_1$,\dots, $\Pi_{n-1}$. 
\end{rem}

\begin{ack}
The authors would like to acknowledge the support of NSF DSS-0729251, NSF-CSSI-0926618, DSS-0240058, the Rice University VIGRE program, and the Defense Advanced Research Projects Agency
under Award No. N66001-10-1-4040. Additionally, research on this projected was supported in part 
by a grant from the Israel Science Foundation and Singapore MOE Tier 2 Grant MOE2011-T2-1-090 (ARC 19/11). Finally, the authors thank the editors and anonymous referees for their time, energy and insight.


\end{ack}



\end{document}